\documentclass[11pt]{article}
\usepackage{amssymb, amsmath,amsthm,amsfonts,a4,bezier}
\usepackage{amscd}
\usepackage{graphpap}
\usepackage[pdftex]{color,graphicx}
\usepackage{enumerate}
\usepackage{mathpazo}
\usepackage{setspace}
\usepackage{comment}
\usepackage{mathtools}
\usepackage{relsize}
\usepackage{indentfirst}
\usepackage{verbatim,bbm}
\usepackage{hyperref}
\usepackage{extarrows}
\usepackage[pdftex]{color}
\usepackage{dutchcal}
\usepackage{tikz}

\DeclareGraphicsExtensions{.jpg,.pdf,.eps,.png}
\newtheorem{theorem}{Theorem}

\newtheorem{definition}{Definition}

\newtheorem{proposition}{Proposition}
\newtheorem{lemma}{Lemma}

\newtheorem{remark}{Remark}

\renewcommand{\leq}{\leqslant}
\renewcommand{\geq}{\geqslant}

\newcommand{\E}{\mathcal{E}}

\newcommand{\B}{\mathcal{B}}

\newcommand{\Z}{{\mathbb{Z}}}

\newcommand{\N}{{\mathbb{N}}}

\renewcommand{\P}{{\mathbb{P}}}

\DeclarePairedDelimiter{\floor}{\lfloor}{\rfloor}
\begin{document}
\thispagestyle{empty}
\baselineskip=28pt
\vskip 5mm
\begin{center} {\LARGE{\bf Exponential decay for Constrained-degree percolation}}
\end{center}

\baselineskip=15pt
\vskip 20mm

\begin{center}\large
Diogo C. dos Santos\footnote{Instituto de Matemática, Universidade Federal de Alagoas, Brazil, \href{mailto:diogo.santos@im.ufal.br}{diogo.santos@im.ufal.br.}}  and Roger W. C. Silva 
 \footnote{Corresponding Author - Departamento de Estat\'{\i}stica,
 Universidade Federal de Minas Gerais, Brazil, \href{mailto:rogerwcs@est.ufmg.br} {rogerwcs@est.ufmg.br.}}
\\ 


\end{center}

\begin{abstract} We consider the Constrained-degree percolation model in random environment (CDPRE) on the square lattice. 
In this model,  each vertex $v$ has an independent random constraint $\kappa_v$ which takes the value $j\in \{0,1,2,3\}$ with probability $\rho_j$. The dynamics is as follows: at time $t=0$ all edges are closed; each edge $e$ attempts to open at a random time $U_e\sim \mathrm{U}(0,1]$, independently of all other edges. It succeeds if at time $U_e$ both  its end-vertices have degrees strictly smaller than their respective constraints.   We obtain exponential decay of the radius of the open cluster of the origin at all times when its expected size is finite. Since CDPRE is dominated by Bernoulli percolation, such result is meaningful only if the supremum of all values of $t$ for which the expected size of the open cluster of the origin is finite is larger than 1/2. We prove this last fact by showing a sharp phase transition for an intermediate model.\\


\end{abstract}

\onehalfspacing
\section{Introduction}\label{intro}
Let $\mathbb{L}^2=(\mathbb{Z}^2,\E)$ be the usual square lattice. To each site $v\in\mathbb{Z}^2$ we associate independently a random variable $\kappa_v$ which takes the value $j\in\{0,1,2,3\}$ with probability $\rho_j$. Denote by $\P_{\rho}$ the corresponding product measure on $\{0,1,2,3\}^{\Z^2}$. Consider the following dependent continuous time percolation process: let $\{U_e\}_{e\in\E}$ be a collection of $i.i.d.$ random variables with uniform distribution on the interval $(0,1]$. At time $t=0$ all edges are closed; each edge $e=\langle u,v\rangle$ opens
 at time $U_e$ if $|\{z\in\mathbb{Z}^2-\{u\}\colon \langle z,u\rangle\mbox{ is open at time }U_e\}|
 < \kappa_{u}$ and $|\{z\in\mathbb{Z}^2-\{v\}\colon \langle z,v\rangle\mbox{ is open at time }U_e\}|
 < \kappa_{v}$. In words, at the random time $U_e$ 
 the edge $e$ attempts to open. It succeeds if both its endpoints have degrees smaller than their respective attached constraints. Once an edge is open, it remains open.
 
The model described above draws inspiration from its deterministic constraint version introduced in \cite{LSSST}. In the deterministic model, constraints are fixed to a constant value $\kappa$ for every vertex. The authors of \cite{LSSST} prove a non-trivial phase transition for the model on $\mathbb{L}^2$ when $\kappa=3$. In contrast, they show absence of percolation when $\kappa=2$, even at time $t=1$. In a recent paper, see \cite{HL}, the authors extend some of the results of \cite{LSSST}, proving a phase transition for the model on $\mathbb{L}^d$, $d\geq 2$, for several values of a constant deterministic $\kappa$. See \cite{Ga,GJ,GL,LZ} for other models with some type of constraint.

The random constraint version we approach in this work was initially introduced in \cite{SSS}. In that work, the authors show a non-trivial phase transition for the model on $\mathbb{L}^2$ when $\rho_3$ is sufficiently large, thus extending the main result of \cite{LSSST}.

A formal definition of the CDPRE model reads as follows. To each edge $e\in\E$ we assign independently a random variable $U_e \sim \mathrm{U}(0,1]$, independent of $\{\kappa_v\}_{v\in\Z^2}$. The corresponding product measure is denoted by $\mathbb{P}$. We think of $U_e$ as the time when edge $e$ attempts to open and usually refer to $\{U_e\}_{e\in\E}$ as a 
configuration of \textit{clocks}. Given a collection of constraints $\kappa=\{\kappa_v\}_{v\in\mathbb{Z}^2}$
and a clock configuration $U=\{U_e\}_{e\in\E}$, let 
\begin{equation*}\omega_t:\{0,1,2,3\}^{\Z^2}\times[0,1]^{\E}\rightarrow\{0,1\}^{\E}
\end{equation*}
be the function that associates the pair $(\kappa,U)$ to a
configuration $\omega_t(\kappa,U)$ of open 
and closed edges at time $t$. From now on, we use the short notation $\omega_t$ and denote by $\omega_{t,e}$ the status of the edge $e$
in the configuration $\omega_t$. We say an edge $e$ is $t$-open ($t$-closed) if $\omega_{t,e}=1$ ($\omega_{t,e}=0$).  Formally, writing ${\mathbbm{1}}_A$ for the indicator function of the event $A$ and $deg(v,t)$ for the degree of vertex $v$ in $\omega_t$, the configuration at edge $e=\langle u,v \rangle$ is written as \begin{equation*}\label{perc_conf}
 \omega_{t,e}={\mathbbm{1}}_{\{U_e\leq t\}}\times{\mathbbm{1}}_{\{\deg(u,U_e)<\kappa_{u}\}}\times{\mathbbm{1}}_{\{\deg(v,U_e)<\kappa_{v}\}}.
\end{equation*}

Using Harris' construction, it is straightforward  to show
that  $\omega_t$ is well defined for almost all sequences $U=\{U_e\}_{e\in\E}$ and $\kappa=\{\kappa_v\}_{v\in\mathbb{Z}^2}$ and all $t\in[0,1]$;  see discussion after Theorem 2 in \cite{SSS}.

Denote by $\P_{\rho,t}$ the pushforward product law governing $\omega_t$, that is, for any measurable set $A\subset\{0,1\}^{\E}$,
\begin{equation*}\label{prod_meas}
\mathbb{P}_{\rho,t}(A)
=(\P_{\rho}\times \P)(\omega_t^{-1}(A)).
\end{equation*}

What makes this model interesting is that, at any fixed time $t>0$, the configurations exhibit infinite-range dependencies. However, as we will show later, the dependence between the states of any two edges decays at least  exponentially as the distance between them increases (see Proposition \ref{log_dec} in Section \ref{sec_exp}), a fact that will be important in this work. 

\begin{remark}We stress that the model lacks the FKG property, which makes the analysis significantly harder. For instance, consider $\rho_3=1$ and $t>0$. Then, the probability that all four edges incident to some vertex $v$ are open at time $t$ vanishes, while the probability that any pair of such edges are open at time $t$ remains strictly positive.
\end{remark}

\subsection{Results and discussion}

Before we state our results, let us introduce some notation. A \textit{path} of $\mathbb{L}^2$ is an alternating sequence $v_0,e_0,v_1,e_1,\dots,e_{n-1},v_n$ of distinct vertices $v_j$ and edges $e_j=\langle v_j,v_{j+1}\rangle$. Such a path has length $n$ and is said to connect $v_0$ to $v_n$. A path is said to be open if all of its edges are open. Write $C_v$ for the open cluster of $v\in\Z^2$, i.e., the set of vertices connected to $v$ by an open path. By translation invariance of the probability measure, we take this vertex to be the origin and define the percolation and susceptibility critical thresholds
\begin{equation*}
t_c(\rho)=\sup\{t\in[0,1]:\P_{\rho,t}(|C|= \infty)=0\},
\end{equation*}
\begin{equation*}
\bar{t}_c(\rho)=\sup\{t\in[0,1]:\mathbb{E}_{\rho,t}(|C|)<\infty\},
\end{equation*}
respectively. Here $\mathbb{E}_{\rho,t}$ denotes expectation with respect to $\P_{\rho,t}$. Clearly, $\bar{t}_c(\rho)\leq t_c(\rho)$.

Let $u,v\in\mathbb{Z}^2$ and denote by $d(u,v)$
the graph distance between $u$ and $v$. Write $B(n)=[-n,n]^2$ for the box of side-length $2n$ centered at the origin. For $x\in\Z^2$, we define $B(x,n)=x+B(n)$. Given $\Gamma\subset\Z^2$, we write $\E(\Gamma)$ to denote the set of edges with both endpoints in $\Gamma$. We use $\partial \Gamma$ to denote the vertex boundary of $\Gamma$, being the set of vertices in $\Gamma$ which are adjacent to some vertex not in $\Gamma$. We also write $\partial^e \Gamma$ for the external edge boundary of $\Gamma$, i.e., the set of edges $e=\langle u,v \rangle$, with $u\in \Gamma$ and $v\notin \Gamma$.

It is not hard to see that the radius of the open cluster of the CDPRE model decays exponentially fast when $t<1/2$.  This follows since the model is stochastically dominated by independent Bernoulli percolation, and the fact that the radius of the open cluster of the latter model decays exponentially fast  (see \cite{AIB} and \cite{DT}) below its critical threshold (see \cite{K}). Theorem \ref{susce} below, whose proof is deferred to Section \ref{sharp}, gives that $\bar{t}_c(\rho)$ (and consequently $t_c(\rho)$) is strictly larger than $1/2$. It is therefore natural to ask: do we have exponential decay for all $t$ smaller than $t_c(\rho)$? We prove  exponential decay of the radius of the open cluster for all $t<\bar{t}_c(\rho)$, giving a partial answer to this question. A nice open problem consists in proving that the model exhibits a sharp phase transition, \textit{i.e.}, that the radius of the open cluster decays exponentially fast for all $t<t_c(\rho)$. In particular, this would give $t_c(\rho)=\bar{t}_c(\rho)$. 

  \begin{theorem}\label{susce} It holds that $\bar{t}_c(\rho)>\frac 12$. 
\end{theorem}

Let $\theta_n(t)$ denote the probability that the origin is connected to $\partial B(n)$ at time $t$. We omit $\rho$ from the notation to keep it clean. We will prove the following theorem.
\begin{theorem}\label{conn2} Let $\rho$ and $t<\bar{t}_c(\rho)$ be given. There exists $\alpha>0$ such that
$$\theta_n(t)\leq e^{-\alpha n},$$ 
for all $n$. 
\end{theorem}

In Section \ref{proof} we will prove these two theorems. The proof of Theorem \ref{susce} is obtained by showing a sharp phase transition for an intermediate model. The proof of Theorem \ref{conn2} consists of an application of a Simon-Lieb type inequality. In Section \ref{remark} we present some final remarks and open problems.

\section{Proofs}\label{proof}

\subsection{Proof of Theorem \ref{conn2}}\label{sec_exp}

To prove Theorem \ref{conn2} we will apply a Simon-Lieb type inequality on boxes of several lengths. Observe that if the origin is connected by an open path to $\partial B(4n)$, then the origin must be connected by an open path to $\partial B(n)$ and there must exist a vertex $w\in\partial B(2n)$ such that $w$ is connected to $\partial B(4n)$ by an open path using edges on the complement of $B(2n)$ only. The main difficulty here is to control the decay of connectivity and the decay of correlations between events whose occurrence depends only on the state of edges inside $B(n)$ and those depending on the state of edges outside $B(2n)$. We observe that the decay of correlations obtained in Theorem 2 of \cite{SSS} is no longer enough here, and we derive a new decay rate which is improved by a $\log n$ factor. 

In what follows, the notation $\{w\xleftrightarrow{\makebox[1.0cm]{\text{\footnotesize{$A$}}}}B\}$ means that the vertex $w$ is connected to some vertex in $B$ using only edges with both endpoints in $A$. All constants $c_1, c_2, c_3, c_4$ appearing in this section are universal and do not depend on $n$ or $t$.

\begin{proposition}\label{log_dec} Fix $m,n\in\mathbb{N}$ such that $2m<n$ and $w\in \partial B(2m)$. Then,
\begin{eqnarray}
\P_{\rho,t}\left(0\longleftrightarrow \partial B(m),w\xleftrightarrow{\makebox[1.0cm]{\text{\footnotesize{$B(2m)^c$}}}}\partial B(n)\right)&\leq&\P_{\rho,t}(0\longleftrightarrow \partial B(m))\P_{\rho,t}\left(w\longleftrightarrow \partial B(n)\right)\nonumber\\
&+& c_1m\exp\left(-\tfrac 12 m\log m\right)\nonumber,
\end{eqnarray} 
for all $\rho=(\rho_0, \rho_1, \rho_2, \rho_3)$ and $t\in[0,1]$.
\end{proposition}

\begin{proof} We follow the argument in Section 2.1 of \cite{AB}. Fix  $t>0$ and let $M_t(x)$ be the set of vertices $y$ such that there is a path $x,e_1,x_1,e_2,x_2\dots,e_k,y$ with $t>U(e_1)>U(e_2)>\dots>U(e_k)$. This gives, with the aid of Stirling's formula, 
\begin{equation}\label{inf_zone}\P_{\rho,t}(M_t(x)\cap\{x+\partial B(m)\}\neq \emptyset)\leq \frac{4\times3^{m-1}}{m!}\leq \tfrac 12\exp\left(-m\log \left(\frac{m}{3e}\right)\right).
\end{equation}

The first inequality in \eqref{inf_zone} holds since, if $\{M_t(x)\cap\{x+\partial B(m)\}\neq \emptyset)\}$ occurs, then there must exist a self-avoiding path of length $m$ starting at  $x$ such that all clocks ring in order before time $t$. 

Write $M_t(\partial B(n))=\bigcup_{x\in \partial B(n)}M_t(x)$ and let $w\in\partial B(2m)$.
Union bound and $\eqref{inf_zone}$ yields
\begin{align*}\P_{\rho,t}\left(M_t(\partial B(m))\cap M_t(w)\neq \emptyset\right)&\leq 2\P_{\rho,t}\left(M_t(\partial B(m))\cap\partial B(\floor{3m/2 })\neq \emptyset\right)\\
&\leq c_1m\exp\left(-\tfrac{1}{2}m\log m\right),
\end{align*} for $m>9e^2$. Note that on  $\{M_t(\partial B(m))\cap M_t(w)= \emptyset\}$, the events $\{0\longleftrightarrow \partial B(m)\}$ and $\{w\xleftrightarrow{\makebox[1.0cm]{\text{\footnotesize{$B(2m)^c$}}}}\partial B(n)\}$ are determined by disjoint sets of edges and hence the covariance vanishes. Therefore, we obtain
$$Cov\left( \mathlarger{\mathlarger{\mathbbm{1}}}\{0\longleftrightarrow \partial B(m)\}, \mathlarger{\mathlarger{\mathbbm{1}}}\{w\xleftrightarrow{\makebox[1.0cm]{\text{\footnotesize{$B(2m)^c$}}}}\partial B(n)\}\right)\leq c_1m\exp\left(-\tfrac{1}{2}m\log m\right).$$ The proof follows by observing that
\begin{align}
\P_{\rho,t}\left(0\longleftrightarrow \partial B(m),\  w\xleftrightarrow{\makebox[1.0cm]{\text{\footnotesize{$B(2m)^c$}}}}\partial B(n)\right)&\leq
\P_{\rho,t}(0\longleftrightarrow \partial B(m))\P_{\rho,t}\left(w\longleftrightarrow \partial B(n)\right)\nonumber\\
&\hspace{-1cm}+Cov\left( \mathlarger{\mathlarger{\mathbbm{1}}}\{0\longleftrightarrow \partial B(m)\}, \mathlarger{\mathlarger{\mathbbm{1}}}\{w\xleftrightarrow{\makebox[1.0cm]{\text{\footnotesize{$B(2m)^c$}}}}\partial B(n)\}\right).\nonumber
\end{align}
\end{proof}

\hspace{-0.6cm}\textit{Proof of Theorem \ref{conn2}.}  Fix $\rho$, $t<\bar{t}_c(\rho)$, and write $\theta_n(t)\equiv\theta_n$. Following the discussion at the beginning of this section, let us consider boxes of side length $2\floor{\sqrt{n}}$. We have
\begin{equation*}
\theta_n\leq \P_{\rho,t}\left(0\longleftrightarrow \partial B(\floor{\sqrt{n}}), \exists w\in \partial B(2\floor{\sqrt{n}})\mbox{ s.t. }\{w\xleftrightarrow{\makebox[1.0cm]{\text{\footnotesize{$B(2\floor{\sqrt{n}})^c$}}}}\partial B(n)\}\right).
\end{equation*}
Applying union bound and then Proposition $\ref{log_dec}$, we have
\begin{equation}\label{main_constr}
\theta_n\leq \theta_{\floor{\sqrt{n}}}\left(\displaystyle\sum_{w\in\partial B(2\floor{\sqrt{n}})}\P_{\rho,t}(w\longleftrightarrow \partial B(n))\right)+c_1n\exp\left(-\tfrac 12\floor{\sqrt{n}}\log \floor{\sqrt n}\right).
\end{equation}
By translation invariance it holds $$\P_{\rho,t}(w\longleftrightarrow \partial B(n))\leq \theta_{n-2\floor{\sqrt{n}}},$$ for any $w\in\partial B(2\floor{\sqrt{n}})$. Hence
\begin{equation*}\label{sl1}
\theta_n\leq 16\floor{\sqrt{n}}\theta_{\floor{\sqrt{n}}}\theta_{n-2\floor{\sqrt{n}}}+c_1n\exp\left(-\tfrac 12\floor{\sqrt{n}}\log \floor{\sqrt n}\right).
\end{equation*}
Iterating the above $\tfrac 12\floor{\sqrt{n}}$ times and using the same argument for $\theta_{n-2j\floor{\sqrt{n}}}$, $j\in\{1,2,\dots,\tfrac 12\floor{\sqrt{n}}\}$, we obtain
\begin{equation}\label{main_eq}
\theta_n \leq \left(c_2\floor{\sqrt{n}}\theta_{\floor{\sqrt{n}}}\right)^{\tfrac12\floor{\sqrt{n}}}+c_1n\exp\left(-\tfrac 12\floor{\sqrt{n}}\log \floor{\sqrt n}\right)\displaystyle\sum_{i=0}^{\floor{\sqrt{n}}-1}\left(c_2\floor{\sqrt{n}}\theta_{\floor{\sqrt{n}}}\right)^i.
\end{equation}

 It is easy to see that, if $\mathbb{E}_{\rho,t}(|C|)<\infty$, then $\sum_{n\geq 1}\theta_n(t)<\infty$. Since $\{\theta_n(t)\}_n$ is decreasing, an exercise in analysis gives 
 \begin{equation}\label{analysis}\lim_{n\rightarrow \infty}n\theta_n(t)=0.
 \end{equation} 
Hence we can find some $n_0\in\mathbb{N}$ such that 
\begin{equation*}\label{psi_bound2}
c_2\floor{\sqrt{n}}\theta_{\floor{\sqrt{n}}}<e^{-2},
\end{equation*} for all $n\geq n_0$. This gives
\begin{equation*}
\theta_n \leq \exp\left(-\sqrt{n}\right)+c_3n\exp\left(-\tfrac 12\floor{\sqrt{n}}\log \floor{\sqrt n}\right),
\end{equation*} for all $n\geq n_0$. Note that 
\begin{eqnarray*}c_2n\exp\left(-\tfrac 12\floor{\sqrt{n}}\log\floor{{\sqrt{n}}}\right)&=&\exp\left\{-\left[\frac{\floor{\sqrt{n}}}{2}+\frac {\left(\log\floor{{\sqrt{n}}}-1\right)\floor{\sqrt{n}}}{2}-\log\left(c_3n\right)\right]\right\}\\
&\leq& \exp\left(-\tfrac 14\sqrt{n}\right),
\end{eqnarray*}
for all $n$ large enough, and hence
\begin{equation}\label{sqrt_bound} \theta_n\leq 2\exp\left(-\tfrac 14 n^{\frac 12}\right).
\end{equation}

The same reasoning yields, for any $n,k\in \N$,
\begin{equation*}
\theta_{2(k+1)n}\leq c_4 n\theta_n\theta_{2kn}+c_4n\exp\left(-\frac12 n\log n\right).
\end{equation*}
We claim that 
\begin{equation}\label{induct}\theta_{2kN}\leq\frac{1}{2^{k}\alpha^k c_4 N^k},
\end{equation} for several but finitely many $k=1,2,\dots,k_{max}$.
Here $\alpha>1$ can be taken any fixed number, \textit{e.g.} $\alpha=2$.  Also, the number $k_{max}$ will be established below, but it suffices that $k_{max}\geq 7$.

Equation \eqref{sqrt_bound} implies the existence of a large fixed $N\in\N$, and of some constant $c_3$ such that 
$$\theta_N\leq \frac{1}{4\alpha c_4 N^2}.$$  Since $\theta_n$ is non-increasing, we have $\theta_{2N}\leq \theta_N$, and \eqref{induct} follows when $k=1$. For all such $k$ that
$$c_4N\exp\left(-\frac12 N\log N \right)\leq \frac12\times\frac{1}{2^{k+1}\alpha^{k+1}c_3N^{k+1}},$$ (solving this we find $k_{max}$) we obtain
\begin{equation*}
\theta_{2(k+1)N}\leq \frac{c_4N}{(4\alpha c_4N^2)(2^k\alpha^kc_4N^k)}+c_4N\exp\left(-\frac12 N\log N\right)\leq \frac{1}{2^{k+1}\alpha^{k+1}c_4 N^{k+1}}.
\end{equation*}
In particular, loosening on the upper bound $\theta_{2kN}\leq 1/(2^k\alpha^kc_4N^k)$, we have that
\begin{itemize}
\item[i)] for $m\in\{2N,4N,\dots,2k_{max}N\}$ we have $\theta_m\leq \alpha^{-m/(2N)}$,
\item[ii)] letting $2k_{max}N=N'$ and $\alpha'=\alpha^{k_{max}}$, we have $\theta_{N'}\leq \frac{1}{4\alpha'c_4(N')^2}.$
\end{itemize}
Item $ii)$ follows since if $k_{max}\geq 7$, then $2^{k_{max}}N^{k_{max}}\geq 4(2k_{max}N)^2$. Repeating the same argument with $N'$ and $\alpha'$ one obtains a new set of values $m\in\{2N',\dots,2k_{max}N'\}$, for which $\theta_m\leq (\alpha')^{-m/(2N')}=\alpha^{-m/(2N)}$. Continuing inductively we obtain an infinite subsequence $m_1,m_2,\dots$, with the property that $m_{j+1}\leq 2m_j$ for all $j\in\N$, such that $\theta_{m_{j}}\leq \alpha^{-m_j/(2N)}$. Since $\theta_n$ is non-increasing in $n$, it follows that $\theta_n\leq \alpha^{-n/(4N)}$, for all $n>N$.
\qed

\subsection{Proof of Theorem \ref{susce}}\label{sharp}

In this section we prove Theorem \ref{susce}. We break the proof in two parts, assuming first that $0<\rho_0<1$. 
Let $\{U(e)\}_{e\in\E}$ and $\{\kappa_v\}_{v\in\Z^2}$ be given.  Define a new percolation configuration $\eta_t$ at edge $e=\langle u,v \rangle$ as
\begin{equation}\label{model}\eta_{t,e}=\left\{\begin{array}{ll}
\mathbbm{1}_{\{U(e)\leq t\}}&\mbox{if}\quad v=u+(0,1)\,,\\
 \mathbbm{1}_{\{U(e)\leq t\}}\mathbbm{1}_{\{\kappa_u\neq0\}}&\mbox{if}\quad v=u+(1,0).
\end{array}\right.
\end{equation}
This corresponds to a percolation model where vertical edges are independently open
with probability $t$ and horizontal edges are independently open with probability $t(1-\rho_0)$.

Let $\widehat{\omega}_t$ denote an independent Bernoulli bond percolation configuration with parameter $t$. Then, according to the terminology used in \cite{AG} and \cite{BBR}, $\eta_t$ is an \textit{essential diminishment} of $\widehat{\omega}_t$, in the sense that there exists a configuration such that $\widehat{\omega}_t(U)$  have a doubly-infinite open path but such that a doubly-infinite open path is not
present after the diminishment is activated at the origin. To see this, take a Bernoulli configuration $\widehat{\omega}_t$ and consider the following rule. To each vertex $u\in\Z^2$, activate a diminishment at $u$ with probability $(1-\rho_0)$. If the diminishment is activated at $u$, then delete the edge $v=u+(1,0)$. This is clearly an essential diminishment and the diminished configuration has the same law as $\eta_t$. Consequently, based on the results in \cite{AG} and \cite{BBR}, the critical threshold for the model \eqref{model} strictly increases and is therefore larger than 1/2. Moreover, due to the stochastic dominance of the random variable $\omega_{t,e}$ by $\eta_{t,e}$, the desired result can be derived from the sharpness of the phase transition observed for independent inhomogeneous Bernoulli percolation (see \cite{AIB}).

We turn to the case $\rho_0=0$. Based on ideas from \cite{LSSST}, we construct an intermediate model  that dominates the CDPRE process when $\rho_0=0$ and is dominated by independent Bernoulli percolation. We will show that the intermediate model phase transition is sharp, which will give us the desired result.   

Let $\Lambda=\{(x_1,x_2)\in\Z^2:x_1=0,1,2,3,4,5 \mbox{ and }x_2=0,1,2,3,4\}$ and $\overline{\Lambda}=\{(x_1,x_2)\in\Z^2:x_1=1,2,3,4 \mbox{ and }x_2=1,2,3\}$. For each $(r,s)\in\Z^2$, define $\Lambda_{r,s}=\Lambda+(6r,5s)$ and $\overline{\Lambda}_{r,s}=\overline{\Lambda}+(6r,5s)$. Consider the following sets of edges in $\E(\Lambda_{r,s})$:
$$g_{r,s}=\langle(6r+2,5s+2),(6r+3,5s+2)\rangle,$$
$$A_{r,s}=\{e\in\E(\overline{\Lambda}_{r,s}):|e\cap\partial\overline{\Lambda}_{r,s}|=1\}.$$
$$B_{r,s}=\E(\Lambda_{r,s})\setminus(g_{r,s}\cup A_{r,s}).$$

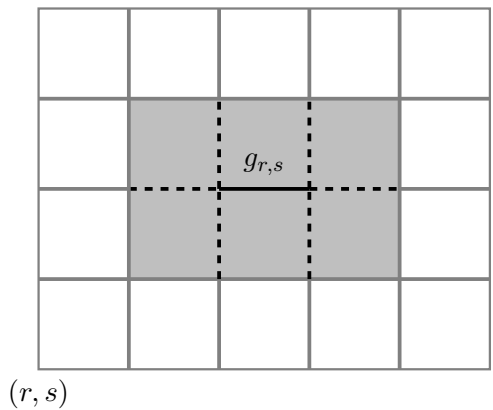
\begin{figure}[h]
\centering
\begin{tikzpicture}[scale=0.03]

\draw[line width=.03cm,gray] (0,0) rectangle (200,160);

\draw[line width=.03cm,gray,fill=lightgray] (40,40) rectangle (160,120);

\draw[line width=0.05cm,gray] (0,40) .. controls (200,40) and (200,40) .. (200,40);
\draw[line width=0.05cm,gray] (0,80) .. controls (40,80) and (40,80) .. (40,80);
\draw[line width=0.05cm,gray] (160,80) .. controls (200,80) and (200,80) .. (200,80);
\draw[line width=0.05cm,gray] (0,120) .. controls (200,120) and (200,120) .. (200,120);
\draw[line width=0.05cm,dashed] (40,80) .. controls (80,80) and (80,80) .. (80,80);
\draw[line width=0.05cm,dashed] (120,80) .. controls (160,80) and (160,80) .. (160,80);

\draw[line width=0.05cm,gray] (40,0) .. controls (40,160) and (40,160) .. (40,160);
\draw[line width=0.05cm,gray] (80,0) .. controls (80,40) and (80,40) .. (80,40);
\draw[line width=0.05cm,gray] (120,0) .. controls (120,40) and (120,40) .. (120,40);
\draw[line width=0.05cm,gray] (80,120) .. controls (80,160) and (80,160) .. (80,160);
\draw[line width=0.05cm,gray] (120,120) .. controls (120,160) and (120,160) .. (120,160);
\draw[line width=0.05cm,gray] (160,0) .. controls (160,160) and (160,160) .. (160,160);
\draw[line width=0.05cm,dashed] (80,40) .. controls (80,80) and (80,80) .. (80,80);
\draw[line width=0.05cm,dashed] (80,80) .. controls (80,120) and (80,120) .. (80,120);
\draw[line width=0.05cm,dashed] (120,40) .. controls (120,80) and (120,80) .. (120,80);
\draw[line width=0.05cm,dashed] (120,80) .. controls (120,120) and (120,120) .. (120,120);

\draw[line width=0.05cm,black] (80,80) .. controls (120,80) and (120,80) .. (120,80);
\draw (0,0)node[below] {$(r,s)$};
\draw (100,100)node[below] {\normalsize $g_{r,s}$};

\begin{scope}[shift={(-33,0)}]
\begin{scope}[shift={(0,50)}]
\draw[line width=.03cm] (10,20) rectangle (10,20);
\draw[line width=.03cm] (10,20) rectangle (10,20);
\end{scope}
\end{scope}
\end{tikzpicture}
\caption{$\Lambda_{r,s}$ (larger box), $\overline{\Lambda}_{r,s}$ (gray box) and the edge $g_{r,s}$. $A_{r,s}$ consists of the dashed edges.}
\label{Lambda}
\end{figure}

The intermediate model is constructed as follows: let $\{U_e\}_{e\in\E}$ be an independent collection of uniform random variables on $[0,1]$ with corresponding product measure $\mathbb{P}$ and define the event 
$$C_{r,s}=\left\{U\in[0,1]^{\E}: \displaystyle\max_{e\in A_{r,s}}U(e)<\displaystyle\min_{e\in\E(\Lambda_{r,s})\setminus A_{r,s}}U(e)\right\}.$$ 
See Figure \ref{Lambda} for a sketch of the boxes and edges involved in the construction. 

A configuration of the intermediate model is a function $$\widetilde{\omega}_t:[0,1]^{\E}\longrightarrow \{0,1\}^{\E}$$ such that
\begin{equation}\label{tilde}\widetilde{\omega}_{t,e}=\left\{\begin{array}{ll}
\mathbbm{1}_{\{U(e)\leq t\}}&\mbox{if}\quad e\notin  \cup_{r,s}\{g_{r,s}\}\,,\\
 \mathbbm{1}_{\{U(e)\leq t\}}\mathbbm{1}_{C^c_{r,s}}&\mbox{if}\quad e=g_{r,s}.
\end{array}\right.
\end{equation}

Note that this has the effect of "diminishing" the percolation configuration by
changing the state of some edges from $t$-open to $t$-closed. We note that there are no constraints in the intermediate model.  Write $\hat{t}_c$ and $\tilde{t}_c$ for the susceptibility critical thresholds (the supremum of $t\in [0,1]$ such that the mean size of the open cluster is finite a.s.) of Bernoulli percolation and the intermediate model, respectively. Note that $\widetilde{\omega}_{t,e}$ can be obtained through a standard coupling (using the same variables $U(e)$) with the CDPRE model. In particular, we obtain $\widetilde{\omega}_{t,e}\geq \omega_{t,e}$ for all $t\in[0,1]$ and for all $e\in\E$, whenever $\rho_0=0$.

 Denoting an independent Bernoulli configuration of parameter $t$ by $\widehat{\omega}_t$, we observe that $\widetilde{\omega}_t$ is an essential diminishment of $\widehat{\omega}_t$.  More precisely, let $W=\{x\in\mathbb{Z}^2: x=(4r+1,3s+1)\mbox{ for some } (r,s)\}$, that is, $W$ consists of those vertices that are left-end points of some $g_{r,s}$, and  consider making a diminishment at each vertex $x\in W$. Let $\eta=(\eta(x): x\in W)$ be a vector lying in the space $\Xi=\{0,1\}^W$, where we  interpret the value $\eta(x)=1$ as meaning that the diminishment at $x$ is activated. Assuming that the random variables $\{\eta(x)\}_{x\in W}$ are i.i.d. Bernoulli with mean $\gamma=1$, then,  when activated, the diminishment acts on $\Lambda_{r,s}$ by deleting the edge $g_{r,s}$ ,whenever $C_{r,s}$ occurs. Therefore, a second application of the main result in \cite{AG} and \cite{BBR} gives that the critical threshold of the intermediate model is strictly larger than $1/2$.

Assuming that the intermediate model phase transition is sharp, we have the inequality  
$$\frac 12= \hat{t}_c< \tilde{t}_c.$$
Since the CDPRE model is dominated by the intermediate model, this gives 
$$\frac 12 < \tilde{t}_c\leq \bar{t}_c(\rho),$$
for all $\rho$ with $\rho_0=0$. 

\begin{remark} We observe that the domination argument described above does not hold when $\rho_0>0$. For instance, suppose $\kappa_{(2,3)}=0$ and $\kappa_{(2,2)}=\kappa_{(3,2)}=3$. If $C_{0,0}$ occurs, then $\widetilde{\omega}_t(g_{0,0})=0$ whilst ${\omega}_{t,g_{0,0}}=1$.  
\end{remark}

Based on the ideas developed in \cite{DRT}, we will prove sharpness of the phase transition for the intermediate model with an application of the OSSS inequality for boolean functions and a suitable randomized algorithm. 

Let us introduce further notation. Assume $I$ is a countable set, and let $(\Omega^I,\pi^{\otimes I})$ be a product probability space with elements denoted by $\omega=(\omega_i)_{i\in I}$. Consider a boolean function $f:\Omega^I\rightarrow \{0,1\}$. An algorithm $\mathbf{T}$ determining $f$ takes a configuration  $\omega$ as an input, and reveals the value of $\omega$ in different coordinates, one by one. At each step, the next coordinate to be revealed depends on the values of $\omega$ revealed so far. This process keeps going until the value of $f$ is the same no matter the values of $\omega$ on the remaining coordinates. For a formal description of a randomized algorithm we refer the reader to \cite{OSSS}.

Denote by $\delta_i(\mathbf{T})$ and $\mathrm{Inf}_i(f)$ the \textit{revealment} and the \textit{influence} of the \textit{i-th} coordinate, respectively, which are defined by
$$\delta_i(\mathbf{T})\coloneqq \pi^{\otimes I}(\mathbf{T}\mbox{ reveals the value of $\omega_i$}),$$
$$\mathrm{Inf}_i(f)\coloneqq \pi^{\otimes I}(f(\omega)\neq f(\omega^*)),$$
where $\omega^*$ is equal to $\omega$ in every coordinate, except the \textit{i-th} coordinate which is resampled independently. The OSSS inequality, introduced in \cite{OSSS} by O'Donnel, Saks, Schramm and Servedio, gives a bound on the variance of $f$ in terms of the influence and the computational complexity of an algorithm for this function. It says that, for any function $f:\Omega^I\rightarrow \{0,1\}$ and any algorithm $\mathbf{T}$ determining $f$,
\begin{equation}\label{OSSS_in}
Var(f)\leq \displaystyle\sum_{i\in I}\delta_i(\mathbf{T})\mathrm{Inf}_i(f).
\end{equation}
 
 The intermediate model is a 5-dependent percolation process and the OSSS inequality can not be directly applied. To overcome this difficulty, we introduce a suitable product space to encode the measure of the intermediate model. We take $\Omega=[0,1]$, $I=\E$ and $\pi^{\otimes I}=\P$. Writing $\mathcal{B}_{n}=\{0\longleftrightarrow \partial B(n)\}$, we are interested in bounding the variance of the boolean function $\mathlarger{\mathlarger{\mathbbm{1}}}_{\widetilde{\omega}_t^{-1}(\mathcal{B}_n)}$ considered as a function from $[0,1]^{\E}$ onto $\{0,1\}$. 

\subsubsection{Bound on the revealment}
Recall the definition of $\widetilde{\omega}_t$ in \eqref{tilde} and denote by $\widetilde{P}_t$ the law of the intermediate model, that is,    
$$\widetilde{P}_t(A)=\mathbb{P}(U\in [0,1]^{\mathcal{E}}\colon \widetilde{\omega}_t(U)\in{A}),$$ for all ${A}\subset\{0,1\}^{\mathcal{E}}$. Write  $\widetilde{\theta}_n(t)=\widetilde{P}_t(\mathcal{B}_n)$ and $S_n(t)=\sum_{k=1}^n \widetilde{\theta}_k(t)$. The next lemma shows the existence of an algorithm determining the boolean function  $\mathlarger{\mathlarger{\mathbbm{1}}}_{\widetilde{\omega}_t^{-1}(\mathcal{B}_n)}$ and gives an upper bound on its revealment. For each $(r,s)\in\Z^2$, write $g_{r,s}=\langle u_{r,s},v_{r,s}\rangle$.

\begin{lemma}\label{reveal_bound} For any $k\in\{0,\dots,n\}$, there exists an algorithm $\mathbf{T}_k$ determining $\mathlarger{\mathlarger{\mathbbm{1}}}_{\widetilde{\omega}_t^{-1}(\mathcal{B}_n)}$ with the property that, for each $e=\langle x_1,x_2\rangle\in\E$,
\begin{equation}\label{reveal_e}\delta_{e}(\mathbf{T}_k)\leq\sum_{i=1,2}\widetilde{P}_t(x_i\leftrightarrow \partial B(k))+\mathlarger{\mathlarger{\mathbbm{1}}}_{\Lambda_{r,s}}(e)\left[\widetilde{P}_t(u_{r,s}\leftrightarrow \partial B(k))+\widetilde{P}_t(v_{r,s}\leftrightarrow \partial B(k))\right].
\end{equation}
\end{lemma}

Once Lemma \ref{reveal_e} is proved, observe that, for any $x\in B(n)$, by summing \eqref{reveal_e} over k, we get
\begin{equation}\label{reveal_sum}
\displaystyle\sum_{k=1}^n\widetilde{P}_t(x_i\leftrightarrow \partial B(k))\leq\sum_{k=1}^n\widetilde{P}_t(x_i\longleftrightarrow \partial B(x_i,d(x_i,\partial B(k))))\leq 2S_n(t),
\end{equation}
where the last inequality follows by translation invariance. Plugging \eqref{reveal_sum} in \eqref{reveal_e} yields

\begin{equation}\label{revealment}\sum_{k=1}^n\delta_{e}(\mathbf{T}_k)\leq \beta S_n(t), 
\end{equation}
for some constant $\beta>0$.

Let $F_n$ denote the set of edges between two vertices within distance $n$ of the origin. We define our algorithm using two growing sequences $\partial B(k)=Z_0\subset Z_1\subset\dots\subset \Z^2$ and $\emptyset=F_0\subset F_1\subset\dots\subset F_n$. At step $m$, we see $Z_m$ as representing the set of vertices that the algorithm found to be connected to $\partial B(k)$, and $F_m$ as the set of edges explored by the algorithm.

\begin{definition}[Algorithm $\mathbf{T}_k$] The algorithm $\mathbf{T}_k$ is defined as follows. Let $e_1, e_2, \dots$ be a fixed ordering of the edges in $E_n$. Write $F_0=\emptyset$ and $Z_0=\partial B(k)$. Assume $Z_m\subset \Z^2$ and $F_m\subset E_n$ are given. 
\begin{enumerate}
\item If there is an edge $e=\langle x,y \rangle\in E_n\setminus F_m$ with $x\in Z_m$ and $y\notin Z_m$, choose the earliest one according to the fixed ordering, set $F_{m+1}=F_m\cup\{e\}$ and write
\begin{equation*}Z_{m+1}=\left\{\begin{array}{ll}
Z_m\cup\{y\}&\mbox{if}\quad \omega_{t,e}=1,\\
 Z_m &\mbox{otherwise}.
\end{array}\right.
\end{equation*}
\item If such $e$ does not exist, write $Z_{m+1}=Z_m$ and $F_{m+1}=F_m\cup\{e\}$.
\end{enumerate}
\end{definition}
Note that, as long as we are in the first case, we are still discovering the connected component of $\partial B(k)$. On the other hand, as soon as we are in the second case, we remain at it. Also, observe that the event where the origin is connected to the boundary of $B(n)$ is already determined before we leave the first case. We are ready to prove Lemma \ref{reveal_bound}.

\medskip

\hspace{-0.6cm}\textit{Proof of Lemma \ref{reveal_bound}.} First, note that the algorithm $\mathbf{T}_k$ discovers the union of all open components of $\partial B(k)$ at time $t$, in particular it determines the function $\mathlarger{\mathlarger{\mathbbm{1}}}_{\widetilde{\omega}_t^{-1}(\mathcal{B}_n)}$. Observe that $e=\langle x,y \rangle\in\Lambda_{r,s}$ is revealed if and only if either $x$, $y$, $u_{r,s}$ or $v_{r,s}$ are connected by a $t$-open path to $\partial B(k)$. Indeed, to determine the status of $g_{r,s}$ all egdes in $\Lambda_{r,s}$ must be revealed. If $e\notin \Lambda_{r,s}$ for all $(r,s)$, then $e$ is revealed if and only if $x$ or $y$ are connected to $\partial B(k)$. This completes the proof.
\qed

\subsubsection{A Russo's type formula}

As before, let $\mathcal{B}_n$ be the event that the origin is connected to the boundary of the box $B(n)$. We have the following Russo's type formula.

\begin{lemma}\label{russo_for} Let $0<\alpha_1<\alpha_2<1$. There exists a constant $q>0$ such that, for all $t\in[\alpha_1,\alpha_2]$, 
$$\frac{d}{dt}\widetilde{P}_t(\mathcal{B}_n)\geq q\sum_{e\in \E(B(n))}\widetilde{P}_t\left(e\mbox{ is pivotal for }\mathcal{B}_n\right).$$ 
\end{lemma}
\begin{proof}
Let $\delta>0$. Then,
\begin{align}\label{russo_1}
    \widetilde{P}_{t+\delta}(\mathcal{B}_n)-\widetilde{P}_t(\B_n)&=\mathbb{P}\left(\widetilde{\omega}_{t+\delta}\in \B_n,\ \widetilde{\omega}_t\notin \B_n\right)\nonumber\\
    &=\mathbb{P}(\widetilde{\omega}_{t+\delta}\in \B_n,\ \widetilde{\omega}_t\notin \B_n,\ \exists\,  e\in \E(B_n) \mbox{ s.t. } t<U(e)\leq t+\delta).
\end{align}
Let $W_{t,\delta}$ be the random set of edges $f$ such that $t<U(f)\leq t+\delta$. Clearly, 
\begin{align}\label{russo_2}
\mathbb{P}(|W_{t,\delta}|\geq 2)=o(\delta).    
\end{align}
From \eqref{russo_1} and \eqref{russo_2} we obtain 
\begin{align*}
  \widetilde{P}_{t+\delta}(\B_n)-\widetilde{P}_t(\B_n)&=\mathbb{P}(\widetilde{\omega}_{t+\delta}\in \B_n,\ \widetilde{\omega}_t\notin \B_n,\ |W_{t,\delta}|=1)+o(\delta)\nonumber\\
  &=\sum_{e\in \E(B(n))}\mathbb{P}\left(\widetilde{\omega}_{t+\delta}\in \B_n ,\ \widetilde{\omega}_t\notin \B_n,\ W_{t,\delta}=\{e\}\right)+o(\delta).
\end{align*}
 We now consider three cases. Remember that $\E(\Lambda_{r,s})=\{g_{r,s}\}\cup A_{r,s}\cup B_{r,s}$.  First, let $e\in\E(B(n))-\bigcup_{r,s}\E(\Lambda_{r,s})$. Then, 
 \begin{align*}
    \mathbb{P}\left( \widetilde{\omega}_{t+\delta}\in \B_n,\ \widetilde{\omega}_t\notin \B_n,\ W_{t,\delta}=\{e\}\right)&=\mathbb{P}(\ e \mbox{ is pivotal for } \B_n \mbox{ in }\widetilde{\omega}_t,\ W_{t,\delta}=\{e\})+o(\delta)\\
    &=\delta\times\mathbb{P}(\ e \mbox{ is pivotal for } \B_n \mbox{ in }\widetilde{\omega}_t)+o(\delta)\\
    &=\delta\times\widetilde{P}_t(\ e \mbox{ is pivotal for } \B_n)+o(\delta).
\end{align*}
Now let $e=g_{r,s}=\langle u_{r,s},v_{r,s}\rangle$ for some pair $(r,s)$. Consider the event $X=\{U(\langle v_{r,s},v_{r,s}+(1,0)\rangle)> t+\delta\}$. This gives the inclusion
    \begin{align*}
    \left\{ \widetilde{\omega}_{t+\delta}\in \B_n,\ \widetilde{\omega}_t\notin \B_n,\ W_{t,\delta}=\{e\}\right\}\supset\left\{{X},\ e \mbox{ is pivotal for }\B_n \mbox{ in }\widetilde{\omega}_t,\ W_{t,\delta}=\{e\}\right\}.
\end{align*}
Note that the event $X\cap\{\ e \mbox{ is pivotal for } \B_n \mbox{ in } \widetilde{\omega}_t\}$ depends only on the variables $U(f)$ with $f\neq g_{r,s}$. Hence
\begin{align*}
    \mathbb{P}(X,\ e \mbox{ is pivotal for }\B_n \mbox{ in }\widetilde{\omega}_t,\ W_{t,\delta}=\{e\})
    &=\mathbb{P}(X,\ e \mbox{ is pivotal for }\B_n \mbox{ in }\widetilde{\omega}_t)\mathbb{P}( W_{t,\delta}=\{e\}).
\end{align*}
Since $\mathbb{P}(X|\ e \mbox{ is pivotal for } \B_n \mbox{ in }\widetilde{\omega}_t)>0$ for all $t\in[\alpha_1,\alpha_2]$, and since the function $t\rightarrow \mathbb{P}(\widetilde{\omega}_t \in A)$ is continuous for any local event $A$, Weierstrass Theorem implies the existence of a constant $M_1>0$ such that 
$$ \mathbb{P}(X,\ e \mbox{ is pivotal for }\B_n \mbox{ in }\widetilde{\omega}_t,\ W_{t,\delta}=\{e\})\geq M_1\delta\times \widetilde{P}_t(\ e \mbox{ is pivotal for } \B_n).$$
Finally, let $e\in A_{r,s}\cup B_{r,s}$ and denote $Y=\{U(g_{r,s})>t\}$. Note that
  \begin{align*}
    \left\{ \widetilde{\omega}_{t+\delta}\in \B_n,\ \widetilde{\omega}_t\notin \B_n,\ W_{t,\delta}=\{e\}\right\}&=\left\{\ e \mbox{ is pivotal for }\B_n \mbox{ in }\widetilde{\omega}_t,\ W_{t,\delta}=\{e\}\right\}\\
    &\supset\left\{Y,\ e \mbox{ is pivotal for }\B_n \mbox{ in }\widetilde{\omega}_t,\ W_{t,\delta}=\{e\}\right\}.
\end{align*}
Note that the event  $Y\cap\{\ e \mbox{ is pivotal for\ }\B_n \mbox{ in } \widetilde{\omega}_t\}$ depends only on the variables $U(f)$ with $f\neq e$. Therefore, as in the previous case, there exists a constant $M_2>0$ such that
\begin{align*}
    \mathbb{P}\left(Y, e \mbox{ is pivotal for } \B_n \mbox{ in }\widetilde{\omega}_t, W_{t,\delta}=\{e\}\right)&=\mathbb{P}(Y, e \mbox{ is pivotal for } \B_n \mbox{ in }\widetilde{\omega}_t) \mathbb{P}(W_{t,\delta}=\{e\})\\
    &\geq M_2\delta\times\widetilde{P}_t(e \mbox{ is pivotal for }\B_n).
\end{align*}
Taking $q=\min\{M_1,M_2\}$ we obtain
$$
    \widetilde{P}_{t+\delta}(\B_n)-\widetilde{P}_t(\B_n)\geq \delta q\sum_{e\in \E(B(n))}\widetilde{P}_t(e\mbox{ is pivotal for }\B_n)+o(\delta).
$$
The result follows by dividing both sides by $\delta$ and  taking the limit when $\delta$ goes to zero.

\end{proof}

\subsubsection{A bound on the influences}

We now seek for a bound on the influence of an edge $e\in \E(B(n))$ on $\mathbbm{1}_{\mathcal{B}_n}$, that is, we seek for a bound on $$\mathrm{Inf}_e(\mathbbm{1}_{\mathcal{B}_n})\coloneqq \P\left(U: \mathbbm{1}_{\mathcal{B}_n}(\widetilde{\omega}_t(U))\neq\mathbbm{1}_{\mathcal{B}_n}(\widetilde{\omega}_t(U^*))\right),$$ where $U$ is equal to $U^*$ in every edge, except edge $e$ which is resampled independently. We do this in two steps. First, assume $e\in\E(\Gamma)-\bigcup_{r,s}\E(\Lambda_{r,s})$ or $e=g_{r,s}$ for some pair $(r,s)$. In this case, the probability that the state of the indicator function change is 
 \begin{align*}\mathrm{Inf}_e(\mathbbm{1}_{\mathcal{B}_n})&=\P\left(U(e)>t,U^*(e)<t, \mbox{ $e$ is pivotal for $\mathcal{B}_n$}\right)+\P\left(U(e)<t, U^*(e)>t, \mbox{ $e$ is pivotal for $\mathcal{B}_n$}\right)\\
&\leq 2\widetilde{P}_t(e\mbox{ is pivotal for }\mathcal{B}_n).
\end{align*}

Now let $e\in \E(\Lambda_{r,s})\setminus g_{r,s}$. In this case, \begin{align*}\mathrm{Inf}_e(\mathbbm{1}_{\mathcal{B}_n})&\leq\P\left(U: \mathbbm{1}_{\mathcal{B}_n}(\widetilde{\omega}_t(U))\neq\mathbbm{1}_{\mathcal{B}_n}(\widetilde{\omega}_t(U^*)), U_{g_{r,s}}>t\right)\\
&+\P\left(U: \mathbbm{1}_{\mathcal{B}_n}(\widetilde{\omega}_t(U))\neq\mathbbm{1}_{\mathcal{B}_n}(\widetilde{\omega}_t(U^*)), U(g_{r,s})\leq t\right).
\end{align*}
If $U(g_{r,s})>t$ and the indicator of $\mathcal{B}_n$ is changed, then $e$ must be pivotal for $\mathcal{B}_n$. If $U(g_{r,s})\leq t$ and the indicator of $\mathcal{B}_n$ is changed, then either $e$ or $g_{r,s}$ must be pivotal for $\mathcal{B}_n$. Putting all together, we obtain 
\begin{equation}\label{bound_influence}
\sum_{e\in B(n)}{\mathrm{Inf}_e}(\mathbbm{1}_{\mathcal{B}_n})\leq \gamma \sum_{e\in B(n)}\widetilde{P}_t(e\mbox{ is pivotal for }\mathcal{B}_n),
\end{equation}
for some constant $\gamma>0$.

Let $t_c^*$ denote the percolation critical threshold for the intermediate model. By stochastic dominance and the results of \cite{LSSST} we know that $1/2 <t_c^*<1$. We now prove that the intermediate model undergoes a sharp phase transition, a fact from which Theorem \ref{susce} is a corollary. 

\begin{theorem}\label{sharpness} Consider the intermediate model on $\Z^2$.
\begin{enumerate}
\item For $t<t_c^*$, there exists $c_t>0$ such that for all $n\geq 1$, $\widetilde{\theta}_n(t)\leq \exp(-c_tn)$.
\item There exists $c>0$ such that for $t>t_c^*$, $\widetilde{P}_t(0\longleftrightarrow \infty)\geq c(t-t_c^*)$.
\end{enumerate}
\end{theorem}
\begin{proof}
Applying the OSSS inequality \eqref{OSSS_in} for each $k$ and then summing on $k$, Equation \eqref{revealment} gives 
\begin{equation*}\label{theta_bound}
\widetilde{\theta}_n(t)(1-\widetilde{\theta}_n(t))\leq \frac{\beta S_n(t)}{n}\sum_{e\in B(n)}\mathrm{Inf_{e}}(\mathbbm{1}_{\mathcal{B}_n}).
\end{equation*}
Equation \eqref{bound_influence} and Lemma \ref{russo_for} give
\begin{equation*}\label{theta_bound2}\sum_{e\in\E(B(n))}\mathrm{Inf_{e}}(\mathbbm{1}_{\mathcal{B}_n})\leq \gamma q^{-1}\frac{d}{dt}\widetilde{\theta}_t(n).
\end{equation*}
Hence, there is a constant $\nu>0$ such that
$$\frac{d}{dt}\widetilde{\theta}_n(t)\geq \frac{\nu n}{ S_n(t)}\widetilde\theta_n(t)(1-\widetilde\theta_n(t)).$$
Fix $t_0\in(t_c^*,\alpha_2)$. Since $\widetilde{\theta}_n(t)$ is increasing in $t$ and $n$, we have $1-\widetilde{\theta}_n(t)\geq 1-\widetilde{\theta}_1(t_0)$ for all $t\leq t_0$. The result follows with an application of Lemma 3 in \cite{DRT} to the function $f_n=\frac{\widetilde{\theta}_n(t)}{\nu (1-\widetilde{\theta}_1(t_0))}$.
\end{proof}

\section{Final remarks}\label{remark}

We finish this paper with a few remarks and also with some unanswered questions. \\

1. Does $t_c(\rho)=\bar{t}_c(\rho)$ hold for any $\rho=(\rho_0, \rho_1, \rho_2, \rho_3)$? \\

One could tackle this problem by showing a sharp phase transition for the CDPRE model, meaning that the radius of the open cluster decays exponentially fast for all $t<t_c(\rho)$. The OSSS method of H. Duminil-Copin, A. Raoufi and V. Tassion (see \cite{DRT} for example) emerges as a promising tool to prove such decay. On one hand, there is a small and well-controlled probability that one needs to look at a far away edge to see what the state of a fixed edge $f$ is (because the sequence of $U(e)$ needs to be decreasing; see also Proposition \ref{log_dec}). Hence, when exploring, it should not be difficult to explore the cluster plus what we need to explore to determine $f$. On the other hand, 
proving a Margulis-Russo type formula seems problematic, given that events of interest are not even monotone in the uniform variables and that the 0-1 variables do not vary nicely in terms of the parameter.\\

2. Does the statement of Theorem \ref{conn2} hold for $d>2$?\\

If we take $d>2$, then there would be an entropy factor of order $n^{d-1}$ in the first term on the r.h.s. of \eqref{main_eq}. In this case we would not have  \eqref{analysis}, which is crucial for our estimate. \\

3. Assume $\rho$ stochastically dominates $\tilde{\rho}$. Does $t_c(\rho)\leq t_c(\tilde{\rho})$ hold?

\section*{Acknowledgements} We are grateful to Rémy Sanchis for several valuable discussions. Diogo C. dos Santos was partially supported by PNPD/CAPES. Roger Silva was partially supported by FAPEMIG (Universal APQ-00774-21). This study was financed in part by the Coordenação de Aperfeiçoamento de Pessoal de Nível Superior – Brasil (CAPES) – Finance Code 001.

\end{document}